\documentclass[12pt,oneside,leqno]{amsart}
\usepackage[margin=1in]{geometry}
\usepackage{amsmath}
\usepackage{amsthm}
\usepackage{amsfonts}
\usepackage{amssymb}
\usepackage{stmaryrd}
\usepackage[hidelinks]{hyperref}
\usepackage{tikz}
\usetikzlibrary{cd}
\usepackage{mathrsfs}

\newenvironment{nouppercase}{%
  \renewcommand{\uppercasenonmath}[1]{}}{}

\usepackage{comment}

\makeatletter
\renewcommand{\eqref}[1]{\textup{(\ignorespaces\ref{#1}\unskip\@@italiccorr)}}
\def\maketag@@@#1{\hbox{\m@th\normalfont\bfseries#1}}
\makeatother

\numberwithin{equation}{section}
\numberwithin{figure}{section}

\theoremstyle{plain}
\newtheorem{thm}[equation]{Theorem}
\newtheorem{lemma}[equation]{Lemma}
\newtheorem{cor}[equation]{Corollary}
\newtheorem{prop}[equation]{Proposition}

\theoremstyle{definition}
\newtheorem{definition}[equation]{Definition}

\newcommand{\R}{\ensuremath \mathbb{R}}

\newcommand{\Z}{\ensuremath \mathbb{Z}}
\newcommand{\N}{\ensuremath \mathbb{N}}

\DeclareMathOperator{\rank}{rank}

\begin{document}

\title[Euclidean domains with no multiplicative norms]{Euclidean domains with no multiplicative norms}

\author{Caleb J.\ Dastrup}
\address{Department of Mathematics, Brigham Young University, Provo, UT 84602, USA}
\email{calebjdastrup@gmail.com}

\author{Pace P.\ Nielsen}
\address{Department of Mathematics, Brigham Young University, Provo, UT 84602, USA}
\email{pace@math.byu.edu}


\begin{abstract}
We construct a Euclidean domain with no multiplicative Euclidean norm to a compatibly well-ordered monoid, and hence with no multiplicative Euclidean norm to $\R$ (under its usual order).

A key step in the proof is showing that the UFD property is preserved when adjoining a free factorization.
\end{abstract}

\begin{nouppercase}
\maketitle
\end{nouppercase}

\section{Introduction}

In this paper, we study the following well-known class of rings.

\begin{definition}
An integral domain $R$ is called a \emph{Euclidean domain} when there exists a function $\varphi\colon R-\{0\}\to \N$ such that for each pair $(a,b)\in R^2$ with $b\neq 0$ and $b\nmid a$, there exist some $q\in R$ satisfying
\[
\varphi(a-qb)<\varphi(b).
\]
Any such function $\varphi$ is called a \emph{Euclidean norm} on $R$.
\end{definition}

Readers may have seen Euclidean domains defined using a slightly different definition, perhaps with the Euclidean norm defined on $0$ (just as some authors allow the zero polynomial to have a degree).  As shown in \cite{AF}, any such minor discrepancy causes no problems.

The historical reason for mapping to $\N$ is to take advantage of the fact that it possesses a well-ordering $<$, which was used in the defining condition for the Euclidean norm.  In the definition above, if we replace $(\N,<)$ by an arbitrary well-ordered set $(I,<)$, we get the class of \emph{transfinite Euclidean domains}, with their corresponding Euclidean norms.  There are transfinite Euclidean domains that possess no $(\N,<)$-valued Euclidean norms.  In fact, for each ordinal $\alpha$, there is a transfinite Euclidean domain such that the images of its Euclidean norms must each have length larger than $\alpha$.  For more information on this topic, the reader is directed to \cite{NielsenEuclid} and the references therein.

It has been a longstanding open problem whether or not every Euclidean domain has a \emph{multiplicative} Euclidean norm, meaning that the image of $\varphi$ is in (a well-ordered subset of) a monoid $(M,\cdot,1)$ and
\[
\varphi(ab)=\varphi(a)\cdot \varphi(b)\ \text{ for all $a,b\in R-\{0\}$}.
\]
As shown recently in \cite[Section 5]{NielsenEuclid}, the answer is no if we force the monoid to be $(\N-\{0\},\cdot,1)$, under its usual ordering.  However, the ring constructed there does have a multiplicative Euclidean norm with values in a well-ordered submonoid of $(\R_{\geq 1},\cdot, 1)$, under its usual ordering.

This raises the question of whether multiplicativity holds for some Euclidean norm, if we allow values in arbitrary monoids.  In a technical---but somewhat trivial---sense the answer is yes, as the following proposition demonstrates.

\begin{prop}
Let $R$ be a Euclidean domain.  There is a well-ordering $\prec$ on $R-\{0\}$ such that with respect to this ordering the identity map on $R-\{0\}$ is a \textup{(}possibly transfinite\textup{)} Euclidean norm for $R$.
\end{prop}
\begin{proof}
Let $\varphi\colon R-\{0\}\to \N$ be any Euclidean norm on $R$.  For each $n\in \N$, fix a well-ordering $\prec_n$ on $\varphi^{-1}(n)$. Define an ordering $\prec$ on $R-\{0\}$ by saying that $x\prec y$ when either
\begin{itemize}
\item[(1)] $x,y\in \varphi^{-1}(n)$, for some $n\in \N$, and $x\prec_n y$, or
\item[(2)] $x\in \varphi^{-1}(m)$ and $y\in \varphi^{-1}(n)$ for some $m,n\in \N$ with $m<n$.
\end{itemize}
The ordering $\prec$ is a well-ordering.

We end by showing that with respect to this ordering, the identity map on $R-\{0\}$ is a Euclidean norm.  Let $a,b\in R$ with $b\neq 0$ and with $b\nmid a$.  There exists some $q\in R$ with $\varphi(a-qb)<\varphi(b)$.  In particular, $a-qb \prec b$.
\end{proof}

There is a significant difference between the well-ordering $\prec$ defined in the previous proposition on the multiplicative monoid $(R-\{0\},\cdot,1)$, and the well-ordering $<$ on the multiplicative monoid $(\N-\{0\},\cdot,1)$.  In the later case, the monoid multiplication is \emph{compatible} with the ordering, in the sense that
\begin{equation}\label{Eq:Compatibility}
a\leq b\Longrightarrow ac\leq bc\ \text{ for all $a,b,c$ in the monoid}.
\end{equation}
(The choice to use the nonstrict order $\leq$ when defining the compatibility condition is deliberate, as it is less stringent.)  The usual ordering on $\R_{\geq 1}$ is also compatible with multiplication.

Do Euclidean domains always have multiplicative Euclidean norms to compatibly well-ordered monoids?  We answer this question in the negative, thus also showing that there are Euclidean domains with no $\R$-valued, multiplicative Euclidean norms (still assuming the values lie in a well-ordered subset, and that $\R$ is given its usual ordering).  The remainder of the paper will be devoted to proving:

\begin{thm}\label{Thm:Main}
There exists a Euclidean domain $R$ that has no multiplicative Euclidean norm to a compatibly well-ordered monoid.
\end{thm}

In Section \ref{Section:UFD} we state some important properties of UFDs that we will need in our construction.  Readers may find Lemma \ref{Lemma:Pumping} of independent interest.  In Section \ref{Section:MainConstruction} we construct an explicit ring that satisfies Theorem \ref{Thm:Main}.  Finally, in Section \ref{Section:MultMono} we generalize the main construction and result slightly.  All rings in this paper are associative, commutative, and unital.  The set $\N$ is assumed to contain $0$.

\section{UFDs and splitting primes}\label{Section:UFD}

The following well-known lemma will be used without further comment.

\begin{lemma}\label{Lemma:PolyUFD}
A polynomial ring in arbitrarily many variables over a UFD is a UFD.
\end{lemma}

Although adjoining polynomial variables causes no problems when working with UFDs, if we adjoin a factor of a prime to a UFD, then the new ring does not need to be a UFD.  The classic example is
\[
\Z\subseteq \Z[\sqrt{-5}].
\]
Surprisingly, when adjoining new factors as freely as possible (in the universal algebra sense), there is no such problem.  To motivate that result, we first discuss some notational conventions that will be used throughout this paper.

Let $R$ be a nonzero ring, let $p\in R$, and let $x$ and $y$ be two algebraically independent polynomial indeterminates.  Consider the ring
\[
R[x,y]/(xy-p).
\]
This is the ring where we have freely adjoined a factorization of $p$.  Given any coset in $R[x,y]/(xy-p)$, it is represented by an $R$-linear combination of $1$, positive powers of $x$, and positive powers of $y$; remove cross terms by repeatedly replacing copies of $xy$ by $p$.  This representative in unique since every nonzero element in the ideal $(xy-p)$ has a cross term.

Treating $x$ as having grade $1$, treating $y$ as having grade $-1$, and treating the elements of $R$ as having grade $0$, we see that $xy-p$ is homogeneous of grade $0$.  Thus, there is an induced $\Z$-grading on $R[x,y]/(xy-p)$.

\begin{lemma}\label{Lemma:Domainness}
If $R$ is an integral domain and $p\in R-\{0\}$, then $R[x,y]/(xy-p)$ is an integral domain.
\end{lemma}
\begin{proof}
Use the standard leading terms argument, under the $\Z$-grading.
\end{proof}

Hereafter, we will refer to integral domains just as \emph{domains}, and we will continue to assume that $R$ is a domain and that $p\neq 0$.  The ring $R$ is naturally isomorphic to a subring of $R[x,y]/(xy-p)$.  Thus, we find it convenient to change notation and work with the isomorphic ring
\[
R':=R[s,s'\, :\, ss'=p]
\]
where $R$ is an actual subring. (This also streamlines the coset notation.)

Throughout this paper we will refer to $s$ and $s'$ as \emph{conjugates}. As $R'$ is a domain, it has a field of fractions.  We may as well identify $s'$ with $p/s$.  Thus, we will often write this ring as
\[
R'=R[s,p/s]\subseteq R[s,s^{-1}].
\]
Hence, elements of $R'$ can be written as Laurent polynomials in $s$, with coefficients from $R$.  However, keep in mind that the conjugates $s$ and $s'$ play dual roles, so we could just as easily have written
\[
R'=R[s',p/s']\subseteq R[s',s'^{-1}].
\]

Since $R'$ is a $\Z$-graded ring, we define the \emph{spread} of a nonzero element $a\in R'$ as the difference between the grade of the leading term of $a$ and the grade of the lowest term of $a$, writing ${\rm spread}(a)$.  Thus, nonzero homogeneous elements are exactly the elements with spread $0$; all other nonzero elements have positive spread.  Since $R'$ is a domain, then given $a,b\in R'-\{0\}$ we have
\begin{equation}\label{Eq:SpreadMultToAdd}
{\rm spread}(ab)={\rm spread}(a)+{\rm spread}(b).
\end{equation}

The following lemma collects some additional basic facts about the ring $R'$, under additional hypotheses.

\begin{lemma}\label{Lemma:Pumping}
Let $R$ be a UFD, and let $p$ be a prime in $R$.  Put $R':=R[s,s'\, :\, ss'=p]$.  The following hold:
\begin{itemize}
\item[\textup{(1)}] Both $s$ and $s'$ are prime in $R'$, and they are not associate.
\item[\textup{(2)}] Any prime $q\in R$ that is not associate to $p$ remains prime in $R'$.
\item[\textup{(3)}] The units in $R'$ are exactly the units in $R$ \textup{(}i.e., $U(R')=U(R)$\textup{)}.  In particular, the associates of any $r\in R$ are the same in both $R$ and $R'$.
\item[\textup{(4)}] The ring $R'$ is a UFD.
\end{itemize}
\end{lemma}
\begin{proof}
(1) By symmetry considerations it suffices to show that $s$ is prime in $R'$ and not associate to $s'$.  We find that
\[
R'/(s)\cong R[x,y]/(xy-p,x)\cong R[y]/(p)\cong (R/(p))[y],
\]
which is a domain since $p$ is prime in $R$.  Further, the image of $s'$ in this domain is nonzero (as it maps to $y$), so $s$ and $s'$ are not associate.

(2) Let $q$ be a prime of $R$ that is not associate to $p$.  We find
\[
R'/(q)\cong (R/(q))[x,y]/(xy-\overline{p}).
\]
The element $\overline{p}\in R/(q)$ is nonzero, since $p$ is prime in $R$ and not associate to $q$.  As $R/(q)$ is a domain, Lemma \ref{Lemma:Domainness} shows that $(R/(q))[x,y]/(xy-\overline{p})$ is a domain.

(3) Let $a,b\in R'$, and assume that $ab=1$.  Since $R'$ is a $\Z$-graded domain, we see that $a$ and $b$ must be homogeneous (of zero spread) and of opposite grades.  If $a$ and $b$ both have grade $0$ we are done, so it suffices to consider the case when $a$ has grade $n\in \N-\{0\}$, and $b$ has grade $-n$.  We then can write $a=r s^{n}$ and $b=r' s'^{n}$ for some nonzero elements $r,r'\in R$.  Now, $rr'p^n=1$ in $R$, which is impossible since $p$ is not a unit in $R$.

(4) First, we will show that every nonzero, nonunit element $a\in R'$ has a factorization into finitely many irreducibles.  In other words, we will show that $R'$ is an atomic domain.

In any factorization of $a$, there are at most ${\rm spread}(a)$ factors with positive spread, by \eqref{Eq:SpreadMultToAdd}.  So we just need to bound the number of homogeneous, nonunit factors in any factorization.  Any homogeneous factor of $a$ is also a factor of the leading term of $a$.  We will show that the leading term is a finite product of primes (rather than merely irreducibles) in $R'$, thus giving the needed bound.

By symmetry considerations, it suffices to consider the case when the leading term of $a$ is $rs^n$, for some $n\in \N$ and some $r\in R-\{0\}$.  Since $R$ is a UFD, we can write
\[
r=uq_1\cdots q_m p^k,
\]
where $u$ is a unit in $R$, the elements $q_1,\ldots, q_m$ are primes in $R$ not associate to $p$ (with $m\in \N$), and $k\in \N$.  Thus, up to a unit, the leading term of $a$ factors as the product $q_1\cdots q_m s^{n+k}s'^{k}$.  These factors are all prime in $R'$, by parts (1) and (2).

We have now shown that $R'$ is an atomic domain.  Also note that
\[
(s)^{-1}R'=R[s,s^{-1}],
\]
which is a UFD, where $s$ is prime in $R'$.  Thus, by Nagata's criterion (see \cite[Theorem 15.61]{ClarkBook}, which generalizes \cite[Lemma 2]{Nagata}) we know that $R'$ is a UFD.
\end{proof}

We want to iteratively use Lemma \ref{Lemma:Pumping}, and the following standard lemma is then useful (cf.\ \cite[Exercise 15.14]{ClarkBook} and \cite[page 7]{Cohn}.)

\begin{lemma}\label{Lemma:UnionUFD}
Let $(I,<)$ be a directed set, and let $(R_i)_{i\in I}$ be a family of UFDs.  Assume for all $i,j\in I$ with $i<j$ that $R_i\subseteq R_j$.  If $R$ is the direct limit of this family, then $U(R)$ is the direct limit of the unit groups $(U(R_i))_{i\in I}$. Further, assuming that for each nonzero, nonunit element $r\in R$, there exists some index $i\in I$ \textup{(}possibly depending on $r$\textup{)} where
\begin{itemize}
\item $r\in R_{i}$, and
\item any prime factor of $r$ in $R_i$ remains prime in $R_j$ for each index $j\geq i$ \textup{(}in other words, prime factorizations eventually stabilize\textup{)},
\end{itemize}
then such factors remain prime in $R$, so $R$ is a UFD.
\end{lemma}
\begin{proof}
Being a direct limit of domains, the set $R$ is also a domain.  The claim about unit groups is also clear.

Given a nonzero, nonunit $r\in R$, fix an index $i\in I$ satisfying the two bullet points.  Write the prime factorization of $r$ in $R_i$ as
\[
r=p_1p_2\cdots p_m.
\]
We will show that each $p_i$ is prime in $R$, thus showing that $R$ is a UFD.  It suffices to show it for $p:=p_1$.

Given $x,y\in R$, assume that $p$ divides $xy$ in $R$.  We may then write $xy=pz$ for some $z\in R$.  There is then some index $j\in I$ with $j\geq i$ such that $r,x,y,z\in R_j$.  Since $p$ is prime in $R_j$, then $p$ is a factor of $x$ or of $y$ in $R_j$.  Thus, the same holds true in $R$.
\end{proof}

An iterative use of Lemma \ref{Lemma:Pumping} must be done wisely, if we hope that the process will result in a UFD.  For example, if we split a prime $p$ into two new primes $s_1$ and $s_1'$, then split $s_1$ into two primes $s_2$ and $s_2'$, and recursively repeat this process infinitely many times, then the resulting ring will not be a UFD. One must avoid infinite chains of nontrivial factors.  We avoid forming such chains by never splitting a given prime more than once.

\begin{prop}\label{Prop:IteratedSplitting}
Let $R$ be a UFD, and let $P:=\{p_i\, :\, i\in I\}$ be a set of nonassociate primes in the ring $R$.  Put $R':=R[s_i,p_i/s_i\, :\, i\in I]$.  The following hold:
\begin{itemize}
\item[\textup{(1)}] For each $i\in I$, both $s_i$ and $p_{i}/s_i$ are prime in $R'$.  Moreover, $s_i$ is not associate to $p_i/s_i$, nor is it associate to either $s_j$ or $p_{j}/s_j$, for each $j\in I$ with $j\neq i$.
\item[\textup{(2)}] Any prime $q\in R$ that is not associate to a prime in $P$ remains prime in $R'$.
\item[\textup{(3)}] The units in $R'$ are exactly the units in $R$.
\item[\textup{(4)}] The ring $R'$ is a UFD.
\end{itemize}
\end{prop}
\begin{proof}
We may as well replace $I$ by an ordinal $\alpha$.  For each ordinal $\beta\leq \alpha$, let
\[
R_{\beta}:=R[s_{i},p_i/s_{i}\, :\, i<\beta].
\]
It suffices to show that each $R_{\beta}$ satisfies the four conditions above, but modified so that $R'$ is replaced with $R_{\beta}$, the index set $I$ is replaced with $\{i\in I\, :\, i<\beta\}$, and the set $P$ is replaced with $\{p_{i}\, :\, i<\beta\}$.  We will work by transfinite induction on $\beta$.

When $\beta=0$, then $R_{\beta}=R$ and all four properties hold.

Next, consider a successor ordinal $\beta+1$, and assume all four properties hold for $R_{\beta}$.  Then they hold for $R_{\beta+1}$ by Lemma \ref{Lemma:Pumping} (taking $R=R_{\beta}$ and $R'=R_{\beta+1}$).

Finally, consider a limit ordinal $\beta>0$, and assume that $R_{\gamma}$ satisfies the four conditions for each $\gamma<\beta$.  Note that $R_{\beta}=\bigcup_{\gamma<\beta}R_{\gamma}$.  Lemma \ref{Lemma:UnionUFD} will give us all four properties for $R_{\beta}$, once we verify that prime factorizations stabilize.  Given any nonzero, nonunit element $a\in R_{\beta}$, then $a\in R_{\gamma}$ for some $\gamma<\beta$.  By our inductive hypothesis, $R_{\gamma}$ is a UFD, and thus $a$ has a prime factorization in $R_{\gamma}$; write it as
\[
a=q_{1}q_2\cdots q_m.
\]
If $q_1$ is not associate to any $p_i$ with $i<\beta$, then using condition (2) and the inductive assumption, $q_1$ is prime in $R_{\delta}$ for each $\gamma\leq \delta<\beta$.  On the other hand, if $q_1$ is associate to some $p_i$ with $i<\beta$, then $q_1$ factors into two primes in $R_{i}$, which remain prime in $R_{\delta}$ for each $i\leq \delta < \beta$, by the inductive assumption and condition (1).  The same factorization process works for $q_j$, for each integer $j\in [1,m]$.  Thus, the prime factorization of $a$ in $R_{\gamma}$ may factor further, but it stabilizes after a finite number of splittings.
\end{proof}

\section{Main construction}\label{Section:MainConstruction}

We will need the following crucial fact about arbitrary Euclidean norms.

\begin{lemma}\label{Lemma:KeyMinPrime}
Let $R$ be a \textup{(}possibly transfinite\textup{)} Euclidean domain, with a Euclidean norm $\varphi$.  If $b\in R-\{0\}$ minimizes $\varphi(b)$ subject to $b\notin U(R)$, then $b$ is prime.
\end{lemma}
\begin{proof}
As $b\in R$ is not a unit, fix a prime $a\in R$ with $a\mid b$.  Given any ``quotient'' $q\in R$, then the ``remainder'' $a-qb$ is also divisible by $a$, hence it is not a unit.  If $a-qb\neq 0$, we cannot have $\varphi(a-qb)<\varphi(b)$, from the minimality assumption.  Thus, from the definition of a Euclidean domain, we must have $b\mid a$, and therefore $b$ is prime (being associate to $a$).
\end{proof}

Still assuming the hypotheses of Lemma \ref{Lemma:KeyMinPrime}, an immediate consequence is that if $p\in R$ is prime with $\varphi(p)$ minimized, then when dividing any numerator by the denominator $p$, we can find a quotient whose corresponding remainder is either zero or a unit.

More can be said when $\varphi$ is a multiplicative norm to a compatibly well-ordered monoid.

\begin{lemma}\label{Lemma:SuboptimalMultComp}
Let $R$ be a \textup{(}possibly transfinite\textup{)} Euclidean domain, with a multiplicative Euclidean norm $\varphi$ to a compatibly well-ordered monoid.  If $p\in R$ is any prime with $\varphi(p)$ minimized, then
\[
\varphi(p^2)\leq \varphi(r)
\]
whenever $r\in R-\{0\}$ has at least two prime factors.  Consequently, when performing the Euclidean division algorithm with denominator $p^2$, there always exists a quotient whose corresponding remainder is zero, a unit, or a prime.
\end{lemma}
\begin{proof}
Let $r\in R-\{0\}$ have at least two prime factors, and write a partial factorization $r=p_1r'$ with $p_1$ prime.  Since $p_1,r'\in R-\{0\}$ are not units, by Lemma \ref{Lemma:KeyMinPrime}  we have $\varphi(p)\leq \varphi(p_1)$ and $\varphi(p)\leq \varphi(r')$.  From the multiplicativity and compatibility assumptions, we find
\[
\varphi(p^2)=\varphi(p)\varphi(p)\leq \varphi(p_1)\varphi(p)\leq \varphi(p_1)\varphi(r')=\varphi(p_1r')=\varphi(r).
\]
The last sentence of the lemma immediately follows.
\end{proof}

To prove Theorem \ref{Thm:Main}, it suffices to find a Euclidean domain $R$ where the last sentence of Lemma \ref{Lemma:SuboptimalMultComp} fails.  Thus, for every prime $p\in R$, we want to be able to choose a special numerator $t\in R$ (depending on $p$) that is not divisible by $p^2$, such that no matter which quotient $q\in R$ we consider, the remainder $t-qp^2$ always has at least two prime divisors.

It is not hard to guarantee the existence of a numerator $t\in R$ not divisible by $p$, and where $t-qp^2$ cannot be a unit.  Thus, the only problem that arises is when $t-qp^2$ itself is prime.  The simple solution is then to apply the methods from the previous section and split any such prime into two new primes.  However, there are two main resulting complications.  First, splitting a prime introduces new elements that can act as quotients.  This problem is overcome by iterating the splitting process recursively.  Second, the resulting ring may fail to be a Euclidean domain.  This problem is handled by introducing a lot of units, making it easier to find remainders with small norms.

We are now prepared to construct an example of a ring $R$ satisfying Theorem \ref{Thm:Main}.  Throughout the remainder of this section, all notations that are introduced will continue to retain their fixed meaning when used in results.

Let $F$ be a field of characteristic zero.  The assumption of characteristic zero will only be used in the proof of Proposition \ref{Prop:EnoughStableDivisors}.

The construction consists entirely of recursively passing to a polynomial ring in multiple indeterminates, followed by passing to a subring of the field of fractions.  All indeterminates are assumed to be independent of one another over $F$.  The indeterminates come in three types with various subscripts.  They are introduced at each successor stage $k+1$ of a recursion, with $k\in \N$.  They are:
\begin{itemize}
\item $s$-type:  These are indeterminates $s_{k+1,p}$, where $p$ runs over of a certain set of primes from the $k$th stage.  The $s$-indeterminates are used to split primes into products of two new prime factors.
\item $t$-type:  These are indeterminates $t_{k+1}$.  The $t$-indeterminates are used to give us access to generic numerators.
\item $u$-type:  These are indeterminates $u_{k+1,a,b}$, where $a$ and $b$  run over certain elements from the $k$th stage.  The $u$-indeterminates help us form certain primes that, at the very end of the construction, we localize into units.  This helps guarantee that the Euclidean algorithm terminates.
\end{itemize}

To begin, let $R_0:=F$. View $R_0$ as the first step in a recursive construction of UFDs
\[
R_0\subseteq R_1\subseteq R_2\subseteq\ldots.
\]
We will put $R_{\infty}:=\bigcup_{k\in \N}R_k$, and the final ring $R$ that we construct will be a localization of the ring $R_{\infty}$.

Some primes in $R_k$ are forced to remain prime at each later stage of the recursion to prevent $R_{\infty}$ from failing to be a UFD, and they also remain prime after the final localization when pasing from $R_{\infty}$ to $R$; these are called the \emph{stable primes}, the set of which is denoted $S_k$.  Some of the primes in $R_k$ will split into two stable primes at the next stage; these are called the \emph{temporary primes}, the set of which is denoted $T_k$.  Finally, some primes will remain prime at each later stage of the recursion, but they will be inverted when passing from $R_{\infty}$ to $R$; these are called the \emph{unit primes}, the set of which is denoted $U_k$.

The sets $S_k$, $T_k$, and $U_k$ partition the primes of $R_k$ (except that we allow empty parts).  We will guarantee that if some prime $p\in R_k$ belongs to one of these three sets, then all its associates belong to the same set.  At the first stage of the recursion take $S_0=T_0=U_0=\emptyset$ (as there are no primes in $F$).

For notational ease, once the set of temporary primes $T_k$ is constructed for a given integer $k\geq 0$, fix (once and for all) a subset $T_k^{\ast}\subseteq T_{k}$ consisting of exactly one prime from each of the association classes in $T_{k}$.

We now explain how the recursion proceeds at successor steps.  Suppose that for some integer $k\geq 0$ we have been given:
\begin{itemize}
\item an $F$-algebra $R_{k}$ that is a UFD, and whose units are $F^{\ast}:=F-\{0\}$, and
\item a partition of the primes of $R_k$ into three sets $S_k$, $T_k$, and $U_k$, each closed under multiplication by $F^{\ast}$.
\end{itemize}
Define $R'_{k+1}$ to be the polynomial ring obtained by adjoining to $R_k$ the following three types of indeterminates:
\begin{itemize}
\item[(1)] $s_{k+1,p}$, for each $p\in T_{k}^{\ast}$,
\item[(2)] $t_{k+1}$, and
\item[(3)] $u_{k+1,a,b}$, for each pair of elements $a,b\in R_k-\{0\}$, where $\gcd_{R_k}(a,b)=1$, and $b$ is not divisible (in $R_k$) by any element of $T_k$, nor by the square of any element in $S_k$.
\end{itemize}
Hereafter, subscripts on these variables will automatically be assumed to belong to the appropriate sets, subject to the restrictions above.

Finally, fix
\[
R_{k+1}:=R'_{k+1}[p/s_{k+1,p}\, :\, p\in T_{k}^{\ast}],
\]
which is a subring of the fraction field of $R'_{k+1}$.   In fact,  by Proposition \ref{Prop:IteratedSplitting}, $R_{k+1}$ is a UFD containing $R_k$, whose units are still exactly $F^{\ast}$.  Moreover, the stable primes in $S_k$ and the unit primes in $U_k$ all stay prime in $R_{k+1}$, and each of the temporary primes in $T_{k}$ has split into two new primes.

For notational convenience, we take $s'_{k+1,p}:=p/s_{k+1,p}$, which is the conjugate of $s_{k+1,p}$.  Now, fix the new set of stable primes to be
\[
S_{k+1}:=S_k \cup (F^{\ast}\cdot \{s_{k+1,p},s'_{k+1,p}\, :\, p\in T_{k}^{\ast}\}).
\]
Also, fix the new set of unit primes to be
\[
U_{k+1}:=U_k\cup (F^{\ast}\cdot \{a-u_{k+1,a,b}b\, :\, u_{k+1,a,b}\in R_{k+1}'\}),
\]
Note that $a-u_{k+1,a,b}b$ really is prime in $R_{k+1}$, since this is a linear polynomial in the indeterminate $u_{k+1,a,b}$, and the two coefficients are nonzero and relatively prime (in $R_k$, and hence also in $R_{k+1}$).  Clearly, the unit primes are not $F^{\ast}$ multiples of the stable primes.  (The unit primes will become units at the very end of the construction.)  Notice that with these definitions, we have two nondecreasing chains
\[
S_0\subseteq S_1\subseteq S_2\subseteq \ldots
\]
and
\[
U_0\subseteq U_1\subseteq U_2\subseteq\ldots
\]
of sets of primes.  Finally, fix $T_{k+1}$ to be the set of all remaining primes in $R_{k+1}$ (and fix a transversal $T_{k+1}^{\ast}$ of the association classes in $T_{k+1}$).

This finishes the recursion.  We now take
\[
R_{\infty}:=\bigcup_{k\in \N}R_k,\qquad S_{\infty}:=\bigcup_{k\in \N}S_k,\qquad U_{\infty}:=\bigcup_{k\in \N}U_k.
\]
By Lemma \ref{Lemma:UnionUFD}, the ring $R_{\infty}$ is a UFD, the set of whose primes is $S_{\infty} \cup U_{\infty}$, which we continue to call the stable primes and unit primes, respectively.

Note that $R_{\infty}$ is a subring of the Laurent polynomial ring in the $s$-type indeterminates, with coefficients that are polynomials over $F$ in the other indeterminates.  Viewed this way, any element $r\in R_{\infty}$ can be written in the unique form $c/d$, where $d$ is a (finite, possibly empty) product of $s$-type indeterminates, $c$ is an $F$-linear combination of distinct monomials in the indeterminates (with nonzero coefficients), and no $s$-type variable that occurs in $d$ also occurs in all the monomials in the support of $c$.

\begin{definition}
Given any $r\in R_{\infty}$, the \emph{rank} of $r$ is the smallest integer $k\geq 0$ such that $r\in R_{k}$, denoted by $\rank(r)$.
\end{definition}

The rank of any indeterminate is its first subscript.  Given any $r\in R_{\infty}-F$, then its rank is easy to determine; it is the largest of all the ranks of the indeterminates that appear in its reduced form.    Given a unit prime $a-u_{k+1,a,b}b$, note that $u_{k+1,a,b}$ is the unique indeterminate of maximal rank that appears in (at least one monomial in the support of) the prime; we will call $u_{k+1,a,b}$ the \emph{ranking indeterminate} of the unit prime.  The following lemma describes a useful fact about ranks.

\begin{lemma}\label{Lemma:Ranks}
The rank of a nonempty product of unit primes is the maximum of the ranks of those primes.
\end{lemma}
\begin{proof}
Consider any nonempty product $\prod_{i=1}^{n}a_i$ of nonzero elements from $R_{\infty}$, for some integer $n\geq 1$.  Write the reduced form for $a_i$ as $c_i/d_i$.  Once $c:=\prod_{i=1}^{n}c_i$ is expanded, it is an $F$-linear combination of distinct monomials (with nonzero coefficients), where the indeterminates in those monomials also appeared in $c_1,c_2,\ldots, c_n$.  Moreover, any indeterminate appearing in some $c_i$ must still appear in at least one monomial in the support of $c$.  Note that $d=\prod_{i=1}^{n}d_i$ is a monomial in $s$-type variables.

The only thing that would prevent $c/d$ from being the reduced form for the product is that some $s$-type variable in the denominator could appear in each of the monomials in the support of the numerator, and cancel off.  The other two types of indeterminates do not cancel off, and in particular if each $a_i$ is a unit prime, the ranking indeterminates all appear in the reduced form.  The largest rank of those ranking indeterminates clearly bounds the ranks of all other indeterminates in the expression $c/d$.
\end{proof}

The previous lemma applies to more general situations.  However, note that it is false when applied to stable primes instead of unit primes, for if $s_{k+1,p}$ is any $s$-type variable, then $p$ factors as $s_{k+1,p}\cdot (p/s_{k+1,p})$, but $p$ has smaller rank than either of the factors.  Of course, the rank of a product can \emph{never} be bigger than the maximum of the ranks of the factors.

Another example of using rank considerations is as follows.  Let $r\in R_{\infty}-F$ have rank $k\geq 1$.  Let $x$ be a $t$-type or $u$-type indeterminate of rank $k$.  Write $r$ in its unique form $c/d$.  Treating $c$ as a polynomial in the indeterminate $x$, we can then write $c=c_0+c_1x + \cdots + c_nx^n$ for some integer $n\geq 0$, where each $c_i$ is a polynomial in the other indeterminates that appear in $c$.  We claim that $c_i/d\in R_{\infty}$ for each integer $i\in [0,n]$.  To see this, note that in the recursive definition of $R_{k}$ we could just as easily have adjoined the single indeterminate $x$ after all the other adjunctions, and so indeed any element of $R_{k}$ is a polynomial in the variable $x$.  When $x$ is an $s$-type indeterminate, we can treat $r$ as a Laurent polynomial in $x$ and again conclude that each coefficient belongs to $R_{\infty}$ (by essentially the same argument).

Note that the result of the previous paragraph may fail when $x$ has rank smaller than $r$.  For example, $1-t_1\in R_1$ is a temporary prime in $R_1$, and so $(1-t_1)/s_{2,1-t_1}$ is an element of $R_2$.  Writing this element as a polynomial in the variable $x:=t_1$, we have
\[
\frac{1}{s_{2,1-t_1}} +\frac{-1}{s_{2,1-t_1}}t_1,
\]
but the coefficient $1/s_{2,1-t_1}$ does \emph{not} belong to $R_2$ (or even $R_{\infty}$).

We are now ready to prove the following key property of the ring $R_{\infty}$.

\begin{prop}\label{Prop:EnoughStableDivisors}
Given any $p,q,v\in R_{\infty}$ with $p\in S_{\infty}$ and $v\in F^{\ast}\cdot \langle U_{\infty}\rangle$, then
\[
vt_{\rank(p)+1}-qp^2
\]
has at least two prime factors from $S_{\infty}$, counting multiplicity.
\end{prop}
\begin{proof}
Fix $k:=\rank(p)$, and fix $r:=vt_{k+1}-qp^2$.  After removing any common (unit prime) factors of $v$ and $q$, we may as well assume $\gcd_{R_{\infty}}(q,v)=1$.  There are three main cases.
\bigskip

\emph{Case 1}: $k+1\geq \max(\rank(v),\rank(q))$.  This means that $t_{k+1}$ cannot appear in any prime factor of $v$ (in $R_{\infty}$), because the corresponding ranking indeterminate would then have rank greater than $k+1$, making $\rank(v)>k+1$ by Lemma \ref{Lemma:Ranks}.

Next, assume by way of contradiction that $t_{k+1}$ does not appear in the reduced form for $r$.  As the reduced form for $vt_{k+1}$ is linear in the indeterminate $t_{k+1}$, the same must be true for $q$.  As $q\in R_{k+1}$, we can then write $q=q_0+q_1 t_{k+1}$, where $q_0,q_1\in R_{k+1}$ do not have $t_{k+1}$ in their reduced forms.  Now, since $t_{k+1}$ does not appear in $r$, we must have $v=q_1 p^2$.  Then $p\mid v$, which is impossible (as $v$ has only unit prime factors).

So, $t_{k+1}$ does appear in the reduced form for $r$.  Therefore, it must appear in one of the prime factors of $r$ in $R_{k+1}$.  It cannot appear in a stable prime (since that would involve an $s$-type variable of larger rank, hence the factor wouldn't belong to $R_{k+1}$).  It cannot appear in a unit prime (since that would involve a $u$-type variable of larger rank).  Thus, it belongs to a temporary prime, which factors into two stable primes at the next stage, and so $r$ has at least two stable prime factors in $R_{\infty}$.
\bigskip

\emph{Case 2}: $k+1<\rank(v)$ and $\rank(q)\leq \rank(v)$.  Fix $\ell:=\rank(v)$.  By Lemma \ref{Lemma:Ranks}, there is a unit prime factor $a-u_{\ell,a,b}b$ of $v$. Fix $u:=u_{\ell,a,b}$, and write $v=(a-ub)^{\alpha}v'$, for some integer $\alpha\geq 1$ and some $v'\in R_{\infty}$ with $(a-ub)\nmid v'$.  (In other words, $\alpha$ is the $(a-ub)$-adic valuation of $v$.)  Notice that $u$ does not appear in $v'$, by rank considerations.

Next, assume by way of contradiction that $u$ does not appear in the reduced form for $r$.  There are two subcases to consider.  First, we may have $p\mid b$.  Notice that from the definition of the $u$-type variables, this means that $p\nmid a$.  Now, treating $v$ as a polynomial in the variable $u$, the coefficient of the linear term is $-\alpha a^{\alpha-1}bv'$.  On the other hand, as $q\in R_{\ell}$, we can write
\[
q=q_0+q_1 u + \text{ higher order terms}
\]
where $q_0,q_1,\ldots \in R_{\ell}$, and $u$ does not appear in the reduced form of any $q_i$.  Thus, since $u$ does not appear in the reduced form for $r$, we must have $p^2q_1 = -\alpha a^{\alpha-1}bv't_{k+1}$.  As we are in characteristic zero, the image of $\alpha$ in $R_{\infty}$ is not zero.  Hence, $p^2$ divides $b$ (since $p\nmid av't_{k+1}$).  This contradicts the fact that the third subscript on a $u$-type variable cannot be divisible by the square of a stable prime.

The second subcase is when $p\nmid b$.  Now, looking at the coefficient of $u^{\alpha}$, rather than the linear coefficient, we obtain a similar contradiction.

Thus, we now know that $u$ does appear in the reduced form for $r$, hence in one of its prime factors in $R_{\ell}$.  It cannot appear in a stable prime, or in a unit prime where $u$ is not the ranking indeterminate (by considering ranks, as in the last paragraph of Case 1).  Also, $a-ub$ cannot divide $r$, else it divides $q$, contradicting the fact that $\gcd_{R_{\infty}}(q,v)=1$.  Thus, $u$ appears in a temporary prime, and we are done as in Case 1.
\bigskip

\emph{Case 3}: $k+1<\rank(q)$ and $\rank(v)<\rank(q)$.  Fix $\ell:=\rank(q)$.  If there is a $t$-type variable of this rank in $q$, then we finish as in the last paragraph of Case 1.  If there is a variable $u:=u_{\ell,a,b}$ of this rank in $q$, then we finish as in the last paragraph of Case 2, except when $r=(a-ub)^{\alpha}r'$, where $r'\in R_{\ell}$ does not have $u$ in its reduced form.  (We must consider this possibility, since we do not have $(a-ub)\mid v$.)  If $p\nmid b$, then considering the coefficient of $u^{\alpha}$ we must have $p\mid r'$.  But then $p\mid r$, and so $p\mid vt_{k+1}$, which is impossible.  If $p\mid b$, then considering the coefficient of $u$, we must have $p\mid r'$, leading to the same contradiction.

Thus, we may reduce to the case that there are no $t$-type nor $u$-type variables of rank $\ell$ in $q$.  Therefore, there is a variable $s:=s_{\ell,p'}$ that must appear in the reduced form for $q$, and hence for $r$.  It must then also appear in at least one of the prime factors of $r$ in $R_{\ell}$.  If such a factor is a temporary prime, we are done as before.  If such a factor is a unit prime, this contradicts the fact that we are working in $R_{\ell}$ since the ranking variable would have rank bigger than $\ell$.  So, we reduce to the case that any such prime factor must be $s$ or its conjugate.  If they both appear as factors, we are done.  So we may assume (by way of contradiction) that one of them does not appear.  By symmetry considerations, it suffices to consider the case when $s$ is a factor of $r$, but not its conjugate.

Writing $q$ as a Laurent polynomial in the indeterminate $s$, then it is of the form
\[
q=s^{-m}(q_0+q_{1}s+\cdots + q_n s^n),
\]
for some integers $m,n\geq 0$, where each $q_i\in R_{\ell}$ does not have $s$ appear in its reduced form (for each integer $i\in [0,n]$), and where $q_n$ is nonzero and $n$ is minimal (so $m$ is also minimal).   If $m>0$, then $s$ appears in the denominator of some prime factor of $r$.  That factor must be a temporary prime by rank considerations (since the conjugate of $s$ is not allowed), which we already ruled out.  Thus $m=0$, and since $s|r$ this forces $vt_{\rank(p)+1}=q_0p^2$, which is impossible since $p$ does not divide the left side.
\end{proof}

As already mentioned, $R_{\infty}$ is a UFD.  We finally invert the unit primes and consider the new ring $R:=U_{\infty}^{-1}R_{\infty}$.  This is a UFD where the primes of $R$ are exactly the associates (in $R$) of the elements in $S_{\infty}$ (i.e., the primes in $R_{\infty}$ that were not inverted).

\begin{cor}
The ring $R$ does not possess a multiplicative Euclidean norm $\varphi$ to a compatibly well-ordered monoid.
\end{cor}
\begin{proof}
Assume, by way of contradiction, that such a map $\varphi$ exists.  Fix a prime $p\in R$ such that $\varphi(p)$ is as small as possible.  Now, since $\varphi$ is a Euclidean norm, there exists some $q\in R$ with
\[
\varphi(t_{\rank(p)+1}-qp^2)<\varphi(p^2).
\]
However, $t_{\rank(p)+1}-qp^2$ has at least two prime factors in $R$, since (after scaling by a unit prime) it has at least two stable prime factors in $R_{\infty}$ by Proposition \ref{Prop:EnoughStableDivisors}.  This contradicts Lemma \ref{Lemma:SuboptimalMultComp}.
\end{proof}

We finish proving Theorem \ref{Thm:Main} by showing:

\begin{thm}\label{Thm:CoolNorm}
The ring $R$ is a Euclidean domain, with an $\N$-valued Euclidean norm.
\end{thm}
\begin{proof}
Define a function $\varphi\colon R-\{0\}\to \N$ by the rule
\[
u\prod_{i=1}^{k}p_i^{\alpha_i}\mapsto \sum_{i=1}^{k}\alpha_i^2,
\]
where the $p_i$ are nonassociate primes in $R$, the $\alpha_i$ are nonnegative integers, and where $u$ is a unit.  Note that $\varphi$ is invariant when multiplying its argument by a unit, but its value increases after multiplying by a (nonzero) nonunit.

We will show that $\varphi$ is a Euclidean norm on $R$.  To that end, let $a,b\in R$ with $b\neq 0$.  Write $a=ga'$ and $b=gb'$, with $a',b'\in R_{\infty}$ and $g\in R$, and with $\gcd_{R_{\infty}}(a',b')=1$.  (In other words, take $g:=\gcd_R(a,b)$, but since GCDs are only defined up to unit multiples, we also guarantee that the cofactors belong to $R_{\infty}$.)  We may then fix an integer $k\geq 0$ such that $a',b'\in R_{k}$.  Note that we can also guarantee that $\gcd_{R_{k}}(a',b')=1$, after absorbing any common unit primes into $g$.  Further, after increasing $k$ by $1$ if necessary, we may assume that all the prime factors of $a'$ and $b'$ from $R$ already live in $R_k$.  There are three cases to consider.
\bigskip

\emph{Case 1}:  $b'\mid a'$ in $R$.  In this case $b\mid a$, and we are done.  For all remaining cases, we will assume $b'\nmid a'$.  In particular, $b'$ is not a unit in $R$ and $a'\neq 0$.
\bigskip

\emph{Case 2}: $b'$ has no repeated prime factor in $R$.  Taking $q:=u_{k+1,a',b'}$, then $a'-qb'$ is a unit of $R$.  Thus, since $b'$ is not a unit in $R$,
\[
\varphi(a-qb)=\varphi((a'-qb')g)=\varphi(g)<\varphi(gb')=\varphi(b),
\]
as desired.
\bigskip

\emph{Case 3}: $b'$ has a repeated prime factor in $R$ (so $\varphi(b')\geq 4>2$).  Taking
\[
\ell:=\max(\rank(a'),\rank(b'),\rank(g))
\]
and putting $q:=t_{\ell+1}$, then $a'-qb'$ is a temporary prime in $R_{\ell+1}$, and hence it factors as a product of two new, conjugate stable primes in $R_{\ell+2}$, which are thus coprime to $g$ by rank considerations. Therefore,
\[
\varphi(a-qb)=\varphi((a'-qb')g)=2+\varphi(g)<\varphi(b')+\varphi(g)\leq \varphi(b'g)=\varphi(b),
\]
where the last (nonstrict) inequality comes from the fact that $\alpha^2+\beta^2\leq (\alpha+\beta)^2$, for all integers $\alpha,\beta\geq 0$.
\end{proof}

Interestingly, this proof also shows that there is a Euclidean algorithm for $R$ that terminates with a zero remainder after at most three steps.  Also note that $R$ has the same cardinality as $F$, which may be any infinite cardinality since $F$ has characteristic zero.

\section{Monotonicity considerations}\label{Section:MultMono}

If $R$ is a Euclidean domain with a multiplicative Euclidean norm $\varphi$ to a compatibly well-ordered monoid, we have
\begin{equation}\label{Eq:MultMon1}
\forall a,b,c\in R-\{0\},\ \varphi(a)\leq \varphi(b) \implies \varphi(ac)\leq \varphi(bc).
\end{equation}
This condition does not, in itself, required that the codomain of $\varphi$ be a monoid.  Thus, we might ask if every Euclidean domain has a Euclidean norm satisfying \eqref{Eq:MultMon1}.  The answer is still no, by using exactly the same ring as constructed in Section \ref{Section:MainConstruction}.  The only change needed in the proof is to modify Lemma \ref{Lemma:SuboptimalMultComp} to work for norms satisfying \eqref{Eq:MultMon1}.

Jesse Elliott recently raised the question of whether every Euclidean domain has a Euclidean norm satisfying the strict version
\begin{equation}\label{Eq:MultMon2}
\forall a,b,c\in R-\{0\},\ \varphi(a)< \varphi(b) \implies \varphi(ac)< \varphi(bc).
\end{equation}
We will show that the answer is still no, even if the conclusion in \eqref{Eq:MultMon2} is weakened to a nonstrict inequality.  We begin with the following modification of Lemma \ref{Lemma:SuboptimalMultComp}.

\begin{lemma}\label{Lemma:Better}
Let $R$ be a \textup{(}possibly transfinite\textup{)} Euclidean domain, with a Euclidean norm $\varphi$ satisfying the property
\[
\forall a,b,c\in R-\{0\},\ \varphi(a)< \varphi(b) \implies \varphi(ac)\leq \varphi(bc).
\]
If $p\in R$ is any prime with $\varphi(p)$ minimized, then
\[
\varphi(p^2)\leq \varphi(r)
\]
whenever $r\in R-\{0\}$ has at least four prime factors.
\end{lemma}
\begin{proof}
Let $r\in R-\{0\}$ have at least four prime factors.  Write $r=p_1p_2p_3r'$, where $p_1$, $p_2$, and $p_3$ are prime.  By Lemma \ref{Lemma:KeyMinPrime}, $\varphi(p)<\varphi(p_1p_2)$ and $\varphi(p)<\varphi(p_3r')$, since $p_1p_2$ and $p_3r'$ are (nonzero) nonunits that are not prime.  Thus, using the given implication twice,
\[
\varphi(p^2)\leq \varphi(p_1p_2p)\leq \varphi(p_1p_2p_3r')=\varphi(r).\qedhere
\]
\end{proof}

We end with the following extension of Theorem \ref{Thm:Main}.

\begin{thm}
There exists a Euclidean domain $R$, with an $\N$-valued Euclidean norm, such that for each \textup{(}possibly transfinite\textup{)} Euclidean norm $\varphi$ on $R$, then there exist $a,b,c\in R-\{0\}$ with $\varphi(a)<\varphi(b)$ and $\varphi(ac)>\varphi(bc)$.
\end{thm}
\begin{proof}
Follow the construction in Section \ref{Section:MainConstruction}, making the following four changes.  First, use Lemma \ref{Lemma:Better} instead of Lemma \ref{Lemma:SuboptimalMultComp}.  Second, instead of splitting a temporary prime $p$ into two stable primes, split it into four stable primes $s_1$, $s_2$, $s_3$, and $p/(s_1s_2s_3)$, treating the first three as new polynomial variables.  (Here, we suppressed the subscripts expressing the dependence of these primes on $p$, and on the rank of $p$, just for readability.)  Continue to call these four primes conjugates.  Third, at the end of Case 3 of the proof of Proposition \ref{Prop:EnoughStableDivisors}, reduce to the situation where the missing conjugate of $s$ is the one with $s$ in the denominator.   Finally, when defining the Euclidean norm in Theorem \ref{Thm:CoolNorm}, use cubes instead of squares (so that the norm of a squared prime is $8$, which is bigger than $4$).

The proofs in Section \ref{Section:MainConstruction} are written in a way that all other needed changes are minor and easily handled.
\end{proof}

\section{Acknowledgements}

We thank Jesse Elliott for pointing out an interesting reference and raising a question that motivated to the work in Section \ref{Section:MultMono}, we thank Luc Guyot for comments that improved the proof of Lemma \ref{Lemma:Pumping}, we thank Kyle Pratt for many comments that improved the quality of this paper, and we thank the anonymous referee for a thorough and detailed report that improved the quality of this paper.   This work was partially supported by a grant from the Simons Foundation (\#963435 to Pace P.\ Nielsen).

\providecommand{\bysame}{\leavevmode\hbox to3em{\hrulefill}\thinspace}
\providecommand{\MR}{\relax\ifhmode\unskip\space\fi MR }
\providecommand{\MRhref}[2]{%
  \href{http://www.ams.org/mathscinet-getitem?mr=#1}{#2}
}
\providecommand{\href}[2]{#2}

\end{document}